\documentclass[11pt]{amsart}   

\usepackage{verbatim}
\usepackage{amssymb,amsthm,amsmath,amsfonts}
\usepackage{color}
\definecolor{darkblue}{rgb}{0, 0, .4}
\definecolor{grey}{rgb}{.7, .7, .7}

\usepackage{ifpdf}
\ifpdf
  \usepackage[pdftex]{epsfig}

  \usepackage{lscape}

  \usepackage[breaklinks]{hyperref}
  \hypersetup{
            colorlinks=true,
            linkcolor=darkblue,
            anchorcolor=darkblue,
            citecolor=darkblue,
            pagecolor=darkblue,
            urlcolor=darkblue,
            pdftitle={},
            pdfauthor={}
  }
\else
  \usepackage[dvips]{epsfig}
  \usepackage{lscape}

  \newcommand{\href}[2]{#2}
  \newcommand{\url}[2]{#2}
\fi

\newtheorem{theorem}{Theorem}[section]
\newtheorem{lemma}[theorem]{Lemma}

\theoremstyle{definition}
\newtheorem{definition}[theorem]{Definition}
\newtheorem{example}[theorem]{Example}

\theoremstyle{remark}

\numberwithin{equation}{section}

\theoremstyle{theorem}
\newtheorem{corollary}[theorem]{Corollary}

\usepackage{graphicx}
\input xy 
\xyoption{all}

\usepackage{tikz}
\usetikzlibrary{trees}
\usepackage{ifthen}

\begin{document}

\title{Positional strategies in games of best choice}

\begin{abstract}
We study a variation of the game of best choice (also known as the secretary
problem or game of googol) under an additional assumption that the ranks of
interview candidates are restricted using permutation pattern-avoidance.  We
describe the optimal positional strategies and develop formulas for the
probability of winning.
\end{abstract}

\author{Aaron Fowlkes and Brant Jones}
\address{Department of Mathematics and Statistics, MSC 1911, James Madison University, Harrisonburg, VA 22807}
\email{\href{mailto:jones3bc@jmu.edu}{\texttt{jones3bc@jmu.edu}}}
\urladdr{\url{http://educ.jmu.edu/\~jones3bc/}}


\date{\today}

\maketitle

\bigskip
\section{Introduction}\label{s:intro}

The game of best choice, also known as the ``secretary problem,'' appeared in
Martin Gardner's 1960 Scientific American column (reprinted in \cite{gardner})
although it has a history which predates this (see e.g. \cite{kadison}).  In
1966, Gilbert and Mosteller \cite{gilbert--mosteller} gave a nice survey of the
problem and solved some variations.  The basic idea is to try to hire the best
candidate out of $N$ applicants for a job, each candidate having a specific
ranking $1$ (worst) through $N$ (best).  When interviewing the candidates, the
decision must be made to hire them or not, on the spot, and candidates cannot be
recalled later.  The order of the interviews is (uniformly) random and so the
interviewer does not know when the top candidate will come in.

As an example, suppose the interviews have rank order $574239618$.  The
interviewer will be able to rank each initial segment of candidates relative to
each other, but will not know their rank overall out of $N$.  So the interviewer
will see
\[ 1, 12, 231, 3421, 45312, 453126, \ldots \]
and must decide when to stop and hire.  We count the game as a win if the best
candidate out of $N$ is hired and as a loss otherwise, with all losses having
equal value.  The optimal strategy, for $N$ sufficiently large, turns out to be
to reject the first $\frac{N}{e}$ of the candidates (about $37\%$) and then
hire the next candidate who is better than all earlier candidates.  

Now, suppose that a consulting firm (with some oracular powers) agrees to
filter candidates for the interviewer.  They offer two strategies.  In the
first strategy, they will guarantee that each time a candidate $B$ ranks lower
than some candidate $A$ already interviewed (``disappointing''), no future
candidates will rank lower than $B$.  In the second strategy, they guarantee
that each time a candidate $B$ ranks higher than some candidate $A$ already
interviewed (``raising the bar''), no future candidates will rank lower than
$A$.  All other aspects of the game remain the same.

Is there any difference between these?  Are they better or worse than the
classical case?  

\bigskip
\section{Refinement}

Interview rank orders are {\bf permutations} of some fixed size $N$ which we
write using the notation $p_1 p_2 \cdots p_N$, where the $p_i$ are the values
$1, 2, \cdots, N$ arranged in some order.  In this work, we restrict the
interview rank orders using pattern avoidance.  

\begin{definition}
We say that the permutation $p = p_1 p_2 \cdots p_N$ contains the pattern $q =
q_1 q_2 q_3$ if there exist $i < j < k$ such that $p_i, p_j, p_k$ are in the
same relative order as $q_1, q_2, q_3$.
\end{definition}

So, the ``disappointment-free'' consulting strategy is equivalent to requiring
the interview rank orders to be $321$-avoiding.  Similarly, the ``bar-raising''
situation is the same as $231$-avoiding.  See the textbook~\cite{bona} for a
gentle introduction to pattern-avoidance.  Putting aside the story about the
consultants, we believe that pattern avoidance is a natural mechanism for
modeling the effect of domain learning by the player during the game.  More
precisely, as the interviewer ranks the current candidates at each step, they
acquire information that allows them to hone the pool to include more relevant
candidates at future time steps.  We represent this honing process using
pattern avoidance.

The {\bf left-to-right maxima} in a permutation $p$ consist of elements $p_j$
that are larger in value than every element $p_i$ to the left (i.e. for $i<j$).
In the game of best choice, it is never optimal to select a candidate that is
not a left-to-right maximum.  A {\bf positional} strategy for the game of best
choice is one in which the interviewer transitions from rejection to hiring
based on the position of the interview.  More precisely, the interviewer may
play the {\bf $k$-positional} strategy on a permutation $p$ by rejecting
candidates $p_1, p_2, \ldots, p_k$ and then accepting the next left-to-right
maximum thereafter.  If $k$ is set too high, it is likely the player will miss
the best candidate.  If $k$ is set too low, they will probably not have set
their standards high enough to capture the best candidate.  We say that a
particular interview rank order is {\bf $k$-winnable} if transitioning from
rejection to hiring after the $k$th interview captures the best candidate.  For
example, $574239618$ is $k$-winnable for $k = 2, 3, 4$, and $5$.  It is
straightforward to verify that a permutation $p$ is $k$-winnable precisely when
$k$ lies between the last two left-to-right maxima in $p$.

In this paper, we restrict to using these positional strategies applied to a
permutation chosen uniformly at random among those avoiding $321$ (or,
alternatively, $231$) in order to facilitate comparison with the classical
case.  For each model, we seek to determine the optimal transition position $k$
and probability of winning, for finite $N$ and asymptotically as $N \rightarrow
\infty$.

We now mention some ties to recent work.  Several authors have investigated the
distribution of various permutation statistics for a random model in which a
pattern-avoiding permutation is chosen uniformly at random.  For example,
\cite{miner-pak} finds the positions of smallest and largest elements as well
as the number of fixed points in a random permutation avoiding a single pattern
of size 3; \cite{mp-312} finds the probability that one or two specified points
occur in a random permutation avoiding 312; and the work of several authors
\cite{dhw,fmw} determines the lengths of the longest monotone and alternating
subsequences in a random permutation avoiding a single pattern of size 3.  
We also consider uniformly random $321$-avoiding and $231$-avoiding
permutations in our work, but the statistics we are concerned with arise from
the game of best choice.  In some sense, our results refine the question of
where a uniformly random pattern-avoiding permutation achieves its maximum
because in our problem we want to transition so as to capture the maximum
value.  We also consider asymptotics for both of our models, so obtaining a
``limit-strategy,'' just as in the classical game.
  
In addition, Wilf has collected some results on distributions of left-to-right
maxima in \cite{wilf} and Prodinger \cite{prodinger} has studied these under a
geometric random model.  Although we phrase our results in terms of the game of
best choice, they may also be viewed as an extension of the literature on
distributions of left-to-right maxima to subsets of pattern-avoiding
permutations.

\bigskip
\section{Raising the bar}

An {\bf extension} of a permutation $p = p_1 p_2 \cdots p_{N-1}$ is the result
of inserting value $N$ into one of the $N$ positions before, between, or after
entries in $p$.

\begin{lemma}\label{l:bar}
Let $p$ be a $231$-permutation of size $N-1$ and $0 \leq k \leq N-1$.  Then
there exists a unique extension of $p$ that is $k$-winnable for $N$. 
\end{lemma}
\begin{proof}
Fix $N$ and $k$.
Let $p_1 p_2 \cdots p_k | p_{k+1} \cdots p_{N-1}$ be a $231$-avoiding permutation of
size $N-1$, with $p_m = \max\{p_1, p_2, \ldots, p_k\}$.

Define $p_w$ to be the leftmost value greater than $p_m$ among $\{p_{k+1},
p_{k+2}, \ldots, p_{N-1}\}$, and let $q$ be the result of inserting $N$ into the
position directly prior to $p_w$ (or into the last position if $p_w$ does not
exist). So we have
\[ q = p_1 p_2 \cdots p_m \cdots p_k | p_{k+1} \cdots p_{w-1} N p_w \cdots
p_{N-1}. \]

We claim that $q$ is the unique $231$-avoiding $k$-winnable extension of
$p$.  To see this, observe that:
\begin{itemize}
    \item By construction, all elements of $\{p_{k+1}, \ldots, p_{w-1}\}$ are
        less than $p_m$, so $q$ is $k$-winnable.
    \item We began with a $231$-avoiding permutation $p$.  If $q$ contains
        $231$, the value $N$ must play the role of ``3.''  Therefore, it
        suffices to show that all of the values lying to the left of $N$ are
        less than all values lying to the right of $N$.  By construction, $p_m =
        \max\{p_1, p_2, \ldots, p_{w-1}\}$ and $p_m < p_w$.  If there exists
        some element $y < p_m$ among the entries $\{p_{w+1}, p_{w+2}, \cdots,
        p_{N-1}\}$ then $(p_m, p_w, y)$ forms a $231$-instance, contradicting
        that $p$ is $231$-avoiding.  Hence, no such $y$ exists and $q$ is
        $231$-avoiding.
    \item If the extension $q$ were not unique, we would have two positions
        $L_1$ and $L_2$, say, where $N$ could be inserted to the right of $p_k$
        to produce distinct $k$-winnable permutations of size $N$.  In
        particular, there must exist at least one element $p_v$ between $L_1$
        and $L_2$.  But the previous paragraph shows that we would require $p_m
        < p_v$ for the extension $q$ using $L_1$ to be $231$-avoiding, so the
        extension using $L_2$ is not $k$-winnable, a contradiction.  Hence, the
        extension is unique.
\end{itemize}
This completes the proof.
\end{proof}

It is well-known that the Catalan numbers $C_N = \frac{1}{N+1}{ {2N} \choose N}$
count the number of $231$-avoiding permutations of size $N$ (see e.g.
\cite{bona}).  Hence, we obtain the following result.  

\begin{corollary}
There are exactly $C_{N-1}$ permutations of size $N$ that are $231$-avoiding and
$k$-winnable.
\end{corollary}
\begin{proof}
For fixed $k$, the set of $231$-avoiding permutations of size $N-1$ are in
bijection with the set of $231$-avoiding $k$-winnable permutations of size $N$
by Lemma~\ref{l:bar}.
\end{proof}

Notice the curious consequence that {\em it does not matter which positional
strategy we use}:  for fixed $N$, the probability of selecting the best
candidate is the same for all $k$.  From the explicit formula, it is
straightforward to work out the asymptotic probability of success $\lim_{N
\rightarrow \infty} \frac{C_{N-1}}{C_N} = \frac{1}{4}$.

\bigskip
\section{Avoiding disappointment}

Next, we consider positional strategies for the $321$-avoiding interview rank
orders.  Recall that a permutation is $k$-winnable if and only if $k$ lies
between its last two left-to-right maxima.  Hence, we study the distribution of
left-to-right maxima in $321$-avoiding permutations.  For this, we make use of
{\bf Dyck paths}.  These may be viewed as paths in the Cartesian plane from
$(0,0)$ to $(N,N)$, consisting of $(0,1)$ steps (i.e.  north) and $(1,0)$ steps
(i.e. east), staying above the line $y = x$.  The {\bf northeast corners} in a
Dyck path consist of a north step immediately followed by an east step.  We
label each northeast corner by the column and height at the end of its east
step.

\begin{example}
The Dyck paths for $N = 3$ are shown below.
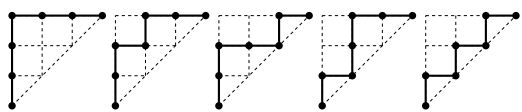
\[	
    \scalebox{0.4}{ \begin{tikzpicture}
        \foreach \i in {0,...,3} { \draw [dashed] (\i,\i) -- (\i,3); }
        \foreach \i in {0,...,3} { \draw [dashed] (0,\i) -- (\i,\i); }
        \draw [dashed] (0,0) -- (3,3);
		\draw [line width=2pt] (0,0) node [circle, fill=black, inner sep=0pt, minimum size=5pt, draw] {} 
		-- ++(0,1) node [circle, fill=black, inner sep=0pt, minimum size=5pt, draw] {}  
		-- ++(0,1) node [circle, fill=black, inner sep=0pt, minimum size=5pt, draw] {}  
		-- ++(0,1) node [circle, fill=black, inner sep=0pt, minimum size=5pt, draw] {}  
		-- ++(1,0) node [circle, fill=black, inner sep=0pt, minimum size=5pt, draw] {}  
		-- ++(1,0) node [circle, fill=black, inner sep=0pt, minimum size=5pt, draw] {}  
		-- ++(1,0) node [circle, fill=black, inner sep=0pt, minimum size=5pt, draw] {}  
		;
		\end{tikzpicture} } 
    \scalebox{0.4}{ \begin{tikzpicture}
        \foreach \i in {0,...,3} { \draw [dashed] (\i,\i) -- (\i,3); }
        \foreach \i in {0,...,3} { \draw [dashed] (0,\i) -- (\i,\i); }
        \draw [dashed] (0,0) -- (3,3);
		\draw [line width=2pt] (0,0) node [circle, fill=black, inner sep=0pt, minimum size=5pt, draw] {} 
		-- ++(0,1) node [circle, fill=black, inner sep=0pt, minimum size=5pt, draw] {}  
		-- ++(0,1) node [circle, fill=black, inner sep=0pt, minimum size=5pt, draw] {}  
		-- ++(1,0) node [circle, fill=black, inner sep=0pt, minimum size=5pt, draw] {}  
		-- ++(0,1) node [circle, fill=black, inner sep=0pt, minimum size=5pt, draw] {}  
		-- ++(1,0) node [circle, fill=black, inner sep=0pt, minimum size=5pt, draw] {}  
		-- ++(1,0) node [circle, fill=black, inner sep=0pt, minimum size=5pt, draw] {}  
		;
		\end{tikzpicture} } 
    \scalebox{0.4}{ \begin{tikzpicture}
        \foreach \i in {0,...,3} { \draw [dashed] (\i,\i) -- (\i,3); }
        \foreach \i in {0,...,3} { \draw [dashed] (0,\i) -- (\i,\i); }
        \draw [dashed] (0,0) -- (3,3);
		\draw [line width=2pt] (0,0) node [circle, fill=black, inner sep=0pt, minimum size=5pt, draw] {} 
		-- ++(0,1) node [circle, fill=black, inner sep=0pt, minimum size=5pt, draw] {}  
		-- ++(0,1) node [circle, fill=black, inner sep=0pt, minimum size=5pt, draw] {}  
		-- ++(1,0) node [circle, fill=black, inner sep=0pt, minimum size=5pt, draw] {}  
		-- ++(1,0) node [circle, fill=black, inner sep=0pt, minimum size=5pt, draw] {}  
		-- ++(0,1) node [circle, fill=black, inner sep=0pt, minimum size=5pt, draw] {}  
		-- ++(1,0) node [circle, fill=black, inner sep=0pt, minimum size=5pt, draw] {}  
		;
		\end{tikzpicture} } 
    \scalebox{0.4}{ \begin{tikzpicture}
        \foreach \i in {0,...,3} { \draw [dashed] (\i,\i) -- (\i,3); }
        \foreach \i in {0,...,3} { \draw [dashed] (0,\i) -- (\i,\i); }
        \draw [dashed] (0,0) -- (3,3);
		\draw [line width=2pt] (0,0) node [circle, fill=black, inner sep=0pt, minimum size=5pt, draw] {} 
		-- ++(0,1) node [circle, fill=black, inner sep=0pt, minimum size=5pt, draw] {}  
		-- ++(1,0) node [circle, fill=black, inner sep=0pt, minimum size=5pt, draw] {}  
		-- ++(0,1) node [circle, fill=black, inner sep=0pt, minimum size=5pt, draw] {}  
		-- ++(0,1) node [circle, fill=black, inner sep=0pt, minimum size=5pt, draw] {}  
		-- ++(1,0) node [circle, fill=black, inner sep=0pt, minimum size=5pt, draw] {}  
		-- ++(1,0) node [circle, fill=black, inner sep=0pt, minimum size=5pt, draw] {}  
		;
		\end{tikzpicture} } 
    \scalebox{0.4}{ \begin{tikzpicture}
        \foreach \i in {0,...,3} { \draw [dashed] (\i,\i) -- (\i,3); }
        \foreach \i in {0,...,3} { \draw [dashed] (0,\i) -- (\i,\i); }
        \draw [dashed] (0,0) -- (3,3);
		\draw [line width=2pt] (0,0) node [circle, fill=black, inner sep=0pt, minimum size=5pt, draw] {} 
		-- ++(0,1) node [circle, fill=black, inner sep=0pt, minimum size=5pt, draw] {}  
		-- ++(1,0) node [circle, fill=black, inner sep=0pt, minimum size=5pt, draw] {}  
		-- ++(0,1) node [circle, fill=black, inner sep=0pt, minimum size=5pt, draw] {}  
		-- ++(1,0) node [circle, fill=black, inner sep=0pt, minimum size=5pt, draw] {}  
		-- ++(0,1) node [circle, fill=black, inner sep=0pt, minimum size=5pt, draw] {}  
		-- ++(1,0) node [circle, fill=black, inner sep=0pt, minimum size=5pt, draw] {}  
		;
		\end{tikzpicture} } 
    \]
Their sets of northeast corners are 
\[ \{ (1,3) \}, \hspace{0.1in} 
 \{ (1,2), (2,3) \}, \hspace{0.1in}
 \{ (1,2), (3,3) \}, \hspace{0.1in}
 \{ (1,1), (2,3) \}, \hspace{0.1in}
 \{ (1,1), (2,2), (3,3) \} \]
 respectively.
\end{example}

\begin{lemma}
The possible sets $\{p_{i_1}, p_{i_2}, \ldots, p_{i_m}\}$ of values and
positions of left-to-right maxima arising from the various permutations of $N$
are in bijection with the sets of northeast corners 
\[ \{(i_j, p_{i_j}) : j = 1, \ldots, m\} \] 
of Dyck paths of size $N$.
\end{lemma}
\begin{proof}
The defining property for a Dyck path is that at each step along the path, the
number of east steps taken so far is less than or equal to the number of
north steps taken so far.  Equivalently, we may consider paths whose northeast
corners satisfy the following two conditions:
\begin{itemize}
    \item There is always a northeast corner in the first column, and 
    \item Whenever we add a northeast corner corresponding to $p_{i_j}$, we take
        at most $p_{i_j} - i_j$ east steps until we reach the next column with a
        northeast corner.
\end{itemize}

But this is precisely equivalent to the conditions that define sets of
left-to-right maxima in a permutation:
\begin{itemize}
    \item The first position is always a left-to-right maximum, and 
    \item Whenever we add a left-to-right maximum corresponding to $p_{i_j}$, we
        have (by definition) at most $p_{i_j} - i_j$ complementary values that
        are smaller than $p_{i_j}$ and have not yet been used.  Hence, there are
        at most $p_{i_j} - i_j$ entries until we reach the next left-to-right
        maximum.
\end{itemize}

Given a Dyck path representing a set of left-to-right maxima, we can produce a
canonical permutation $p$ that realizes this set of left-to-right maxima as
follows:  Place each $p_{i_j}$ into position $i_j$ and then fill the
complementary positions with the complementary values $\{1, 2, \ldots, N\}
\setminus \{p_{i_1}, \ldots, p_{i_m}\}$ arranged increasingly.  In terms of the
Dyck path, we can label the northeast corners by the value of their
corresponding left-to-right maximum, and then label the remaining horizontal
edges with the complementary values, arranged increasingly as we read north and
east along the path.  Thus, the label for column $i$ of the Dyck path gives the
value for the $i$th position of the permutation.  
\end{proof}

As an example in $N = 8$, if $p_1 = 4, p_3 = 7$, and $p_5 = 8$ are the
$p_{i_j}$, we obtain $p = 41728356$; this is illustrated in
Figure~\ref{f:lrmex}.

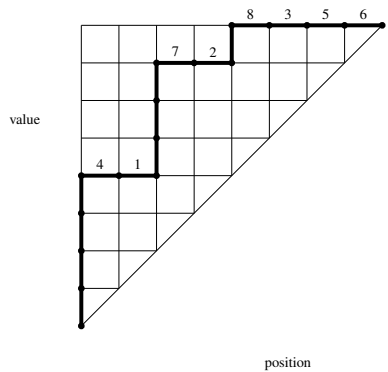
\begin{figure}[ht]
		\[ \scalebox{0.5}{ \begin{tikzpicture}
			\foreach \i in {0,...,8} { \draw (\i,\i) -- (\i,8); }
			\foreach \i in {0,...,8} { \draw (0,\i) -- (\i,\i); }
			\draw (0,0) -- (8,8);
			\draw [line width=3pt] (0,0) node [circle, fill=black, inner sep=0pt, minimum size=2pt, draw] {} 
			-- ++(0,1) node [circle, fill=black, inner sep=0pt, minimum size=2pt, draw] {}  
			-- ++(0,1) node [circle, fill=black, inner sep=0pt, minimum size=2pt, draw] {}  
			-- ++(0,1) node [circle, fill=black, inner sep=0pt, minimum size=2pt, draw] {}  
			-- ++(0,1) node [circle, fill=black, inner sep=0pt, minimum size=2pt, draw] {}  
			-- ++(1,0) node [circle, fill=black, inner sep=0pt, minimum size=2pt, draw] {}  
			-- ++(1,0) node [circle, fill=black, inner sep=0pt, minimum size=2pt, draw] {}  
			-- ++(0,1) node [circle, fill=black, inner sep=0pt, minimum size=2pt, draw] {}  
			-- ++(0,1) node [circle, fill=black, inner sep=0pt, minimum size=2pt, draw] {}  
			-- ++(0,1) node [circle, fill=black, inner sep=0pt, minimum size=2pt, draw] {}  
			-- ++(1,0) node [circle, fill=black, inner sep=0pt, minimum size=2pt, draw] {}  
			-- ++(1,0) node [circle, fill=black, inner sep=0pt, minimum size=2pt, draw] {}  
			-- ++(0,1) node [circle, fill=black, inner sep=0pt, minimum size=2pt, draw] {}  
			-- ++(1,0) node [circle, fill=black, inner sep=0pt, minimum size=2pt, draw] {}  
			-- ++(1,0) node [circle, fill=black, inner sep=0pt, minimum size=2pt, draw] {}  
			-- ++(1,0) node [circle, fill=black, inner sep=0pt, minimum size=2pt, draw] {}  
			-- ++(1,0) node [circle, fill=black, inner sep=0pt, minimum size=2pt, draw] {}  
			;
			\draw (5.5,-1) node {position};
			\draw (-1.5,5.5) node {value};
			\draw (0.5,4.3) node {4};
			\draw (1.5,4.3) node {1};
			\draw (2.5,7.3) node {7};
			\draw (3.5,7.3) node {2};
			\draw (4.5,8.3) node {8};
			\draw (5.5,8.3) node {3};
			\draw (6.5,8.3) node {5};
			\draw (7.5,8.3) node {6};
			\end{tikzpicture} }\]
\caption{Completing the set of left-to-right maxima $\{p_1 = 4, p_3 = 7, p_5 =
8\}$}\label{f:lrmex}
\end{figure}

Recall that the Catalan numbers $C_N$ count $321$-avoiding permutations of size
$N$, and also count the number of Dyck paths of size $N$ (see e.g. \cite{bona}).
Hence, we obtain the following result.

\begin{corollary}\label{c:path}
A $321$-avoiding permutation $p$ of size $N$ is uniquely determined by the
values and positions of its left-to-right maxima.  
\end{corollary}
\begin{proof}
The construction in the previous proof produces $C_N$ distinct permutations of
size $N$ that have the structure of two increasing sequences shuffled together
(namely, the sequence of left-to-right maxima, and the sequence of
complementary values).  Hence, the permutations constructed from Dyck paths in
the previous result are all $321$-avoiding.  Since there are Catalan many of
each, there must be exactly one $321$-avoiding permutation for each Dyck path.
\end{proof}

\begin{figure}[ht]
	\[ \scalebox{0.65}{ \begin{tikzpicture}
		\foreach \i in {0,...,8} { \draw (\i,\i) -- (\i,8); }
		\foreach \i in {0,...,8} { \draw (0,\i) -- (\i,\i); }
		\draw (0,0) -- (2,2);
		\draw [dashed] (2,2) -- (3,3);
		\draw (3,3) -- (8,8);
		\draw [line width=2pt] (3,4) node [circle, fill=black, inner sep=0pt,
        minimum size=3pt, draw] {} -- (4,4) node [circle, fill=black, inner sep=0pt, minimum size=3pt, draw] {};
		\draw [line width=2pt] (3,5) node [circle, fill=black, inner sep=0pt,
        minimum size=3pt, draw] {} -- (4,5) node [circle, fill=black, inner sep=0pt, minimum size=3pt, draw] {};
		\draw [line width=2pt] (3,6) node [circle, fill=black, inner sep=0pt,
        minimum size=3pt, draw] {} -- (4,6) node [circle, fill=black, inner sep=0pt, minimum size=3pt, draw] {};
		\draw [line width=2pt,color=blue] (3,7) node [circle, fill=black, inner sep=0pt, minimum size=3pt, draw] {} -- (4,7) node [circle, fill=black, inner sep=0pt, minimum size=3pt, draw] {};
		\draw (-0.2,-0.2) node {$(0,0)$};
		\draw (8.2,8.2) node {$(N,N)$};
		\draw (3.5,8.2) node {${ }^{k=N-4}$};
        \draw (3.25,7.23) node {$T_4$};
        \draw (4.25,7.23) node {$T_3$};
        \draw (5.25,7.23) node {$T_2$};
        \draw (6.25,7.23) node {$T_1$};
        \draw (5.38,6.23) node {$\Delta T_1$};
        \draw (4.38,6.23) node {$\Delta T_2$};
        \draw (3.38,6.23) node {$\Delta S_3$};
		\end{tikzpicture} } \]
        \caption{Schematic for path recurrences}\label{f:sch}
\end{figure}
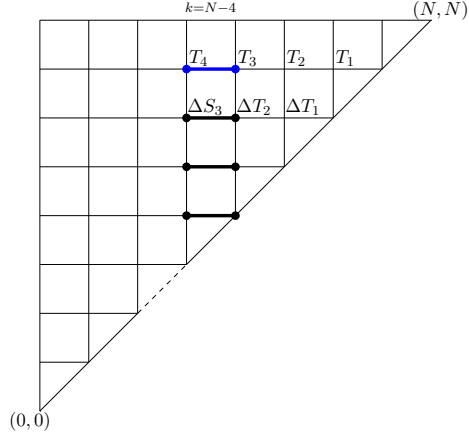

\begin{definition}
For $1 \leq i \leq N-1$ define $T_i(N)$ to be the total number of partial Dyck
paths from $(0,0)$ to $(N-1-i,N-1)$, and define $S_i(N)$ to be the number of
Dyck paths from $(0,0)$ to $(N,N)$ where column $N-i$ lies weakly right of the
next-to-last northeast corner and strictly left of the last northeast corner in
the path.
\end{definition}

By Corollary~\ref{c:path}, the $S_i(N)$ are the number of $(N-i)$-winnable
permutations of $N$.  For example, the path in Figure~\ref{f:lrmex} would be
counted in $S_i(N)$ for $N-i \in \{3, 4\}$ because the last two northeast
corners occur in columns $3$ and $5$, respectively.  Some initial values are
given in Figure~\ref{f:nums}.  If we divide by the $N$th Catalan number we
obtain the probability of success for the corresponding $(N-i)$-positional
strategy.  These are illustrated in Figure~\ref{f:pcts}.  It turns out that the
$T_i(N)$ are Catalan triangle entries at $(N-1, i)$, namely $T_i(N) =
\frac{i+1}{N}{ {2(N-1)-i} \choose {N-1}}$, but we do not use this in our
development.

Now, define an operation $\Delta$ that acts on a function of $N$ by replacing
$N$ with $N-1$.  That is, $\Delta f(N) = f(N-1)$.  We prefer to use this
operator, with the argument $N$ suppressed, as a notational convenience
for our formulas and figures (although all of our results can be obtained
without it).  We next prove recurrences for the $S_i$ and $T_i$ that will
facilitate their computation.

\begin{theorem}\label{t:rec}
We have \[ T_{i} = T_{i-1} - \Delta T_{i-2} \]
with $T_1 = C_{N-1}$ and $T_2 = C_{N-1} - C_{N-2}$ and
\[ S_{i} = i\ T_{i} + \Delta S_{i-1} \]
with $S_1 = C_{N-1}$.
\end{theorem}
\begin{proof}
See Figure~\ref{f:sch} for a schematic illustrating these recurrences.

The recurrence for $T$ follows because each path counted by $T_{i-1}(N)$ must have
ended with a vertical step or a horizontal step; these are counted by $\Delta
T_{i-2}(N) = T_{i-2}(N-1)$ and $T_i(N)$, respectively.

The recurrence for $S$ follows because each path counted by $S_{i}(N)$ passes
through column $N-i$ at level $N-1$ or passes through column $N-i$ below level
$N-1$.  The first set of paths is counted by $i T_i(N)$ because any path ending
at $(N-1-i,N-1)$ can be extended in $i$ ways depending on which of the columns
$N-i, N-i+2, \ldots, N-1$ is used for the last vertical step.  The second set
of paths is counted by $\Delta S_{i-1}(N) = S_{i-1}(N-1)$ because we can
bijectively extend any path passing the required column and ending at
$(N-1,N-1)$ to end at $(N,N)$ instead by inserting one more pair of
vertical/horizontal steps at the last northeast corner.
\end{proof}

Using this theorem, we may write each $S_i$ and $T_i$ as a linear combination of
Catalan numbers.  On the one hand, applying $\Delta$ to $S_i$, say, simply
restricts the Dyck paths we are counting to end at $(N-1,N-1)$ instead of
$(N,N)$.  Algebraically, applying $\Delta$ replaces each Catalan number in the
linear combination with the previous Catalan number.

\begin{example}\label{e:forms}
Applying the recurrences from Theorem~\ref{t:rec}, we have
\[ T_3 = (C_{n-1} - C_{n-2}) - \Delta(C_{n-1}) = C_{n-1} - 2 C_{n-2} \]
\[ T_4 = ( C_{n-1} - 2 C_{n-2} ) - \Delta(C_{n-1}-C_{n-2}) = C_{n-1} - 3 C_{n-2} + C_{n-3} \]
\[ T_5 = ( C_{n-1} - 3 C_{n-2} + C_{n-3} ) - \Delta(C_{n-1} - 2 C_{n-2} ) = C_{n-1} - 4 C_{n-2} + 3 C_{n-3} \]
and
\[ S_2 = 2 (C_{n-1} - C_{n-2}) + \Delta(C_{n-1}) = 2 C_{n-1} - C_{n-2} \]
\[ S_3 = 3 (C_{n-1} - 2 C_{n-2}) + \Delta( 2 C_{n-1} - C_{n-2} ) = 3 C_{n-1} - 4 C_{n-2} - C_{n-3} \]
\[ S_4 = 4 (C_{n-1} - 3 C_{n-2} + C_{n-3}) + \Delta(  3 C_{n-1} - 4 C_{n-2} - C_{n-3} ) = 4 C_{n-1} - 9 C_{n-2} - C_{n-4} \]
\[ S_5 = 5 (C_{n-1} - 4 C_{n-2} + 3 C_{n-3}) + \Delta( 4 C_{n-1} - 9 C_{n-2} - C_{n-4}  ) = 5 C_{n-1} - 16 C_{n-2} + 6 C_{n-3} - C_{n-5}. \]
\end{example}

\setcounter{MaxMatrixCols}{30}
\begin{figure}[ht]
    \[
    \tiny
        \begin{matrix}
			N\backslash k  & -11 & -10 & -9 & -8 & -7 & -6 & -5 & -4 & -3 & -2 & -1 \\ 
			& & & & & & & & & & & \\ 
			2  &  & & & & & & & & & & \bf 1 \\ 
			3  &  & & & & & & & & & \bf 3 & 2 \\ 
			4  &  & & & & & & & & 6 & \bf 8 & 5 \\ 
			5  &  & & & & & & & 10 & 20 & \bf 23 & 14 \\ 
			6  &  & & & & & & 15 & 40 & 65 & \bf 70 & 42 \\ 
			7  &  & & & & & 21 & 70 & 145 & 214 & \bf 222 & 132 \\ 
			8  &  & & & & 28 & 112 & 280 & 514 & 717 & \bf 726 & 429 \\ 
			9  &   &  &  & 36 & 168 & 490 & 1064 & 1817 & \bf 2442 & 2431 & 1430 \\ 
			10  & & & 45 & 240 & 798 & 1988 & 3962 & 6446 & \bf 8437 & 8294 & 4862 \\ 
			11  & & 55 & 330 & 1230 & 3444 & 7784 & 14636 & 22997 & \bf 29510 & 28730 & 16796 \\ 
			12  & 66 & 440 & 1815 & 5628 & 14154 & 29924 & 53937 & 82550 & \bf 104312 & 100776 & 58786 \\ 
    \end{matrix} 
\]
\caption{Number of $k$-winnable $321$-avoiding permutations of $N$}\label{f:nums}
\end{figure}

\begin{figure}[ht]
\[ 
    \tiny
        \begin{matrix}
			N\backslash k  & -11 & -10 & -9 & -8 & -7 & -6 & -5 & -4 & -3 & -2 & -1 \\ 
			& & & & & & & & & & & \\ 
			2 &  &   &   &   &   &   &   &   &   &   &  \bf 50.0  &   \\ 
			3 &  &   &   &   &   &   &   &   &   &  \bf 60.0  &  40.0  &   \\ 
			4 &  &   &   &   &   &   &   &   &  42.8  &  \bf 57.1  &  35.7  &   \\ 
			5 &  &   &   &   &   &   &   &  23.8  &  47.6  &  \bf 54.7  &  33.3  &   \\ 
			6 &  &   &   &   &   &   &  11.3  &  30.3  &  49.2  &  \bf 53.0  &  31.8  &   \\ 
			7 &  &   &   &   &   &  4.89  &  16.3  &  33.7  &  49.8  &  \bf 51.7  &  30.7  &   \\ 
			8 &  &   &   &   &  1.95  &  7.83  &  19.5  &  35.9  &  50.1  &  \bf 50.7  &  30.0  &   \\ 
			9 &  &   &   &  0.74  &  3.45  &  10.0  &  21.8  &  37.3  &  \bf 50.2  &  50.0  &  29.4  &   \\ 
			10 &  &   &  0.26  &  1.42  &  4.75  &  11.8  &  23.5  &  38.3  &  \bf 50.2  &  49.3  &  28.9  &   \\ 
			11 &  &  0.09  &  0.56  &  2.09  &  5.85  &  13.2  &  24.8  &  39.1  &  \bf 50.1  &  48.8  &  28.5  &   \\ 
			12 & 0.03  &  0.21  &  0.87  &  2.7  &  6.8  &  14.3  &  25.9  &  39.6  &  \bf 50.1  &  48.4  &  28.2  &   \\ 
			13 & 0.07  &  0.34  &  1.18  &  3.26  &  7.61  &  15.3  &  26.7  &  40.1  &  \bf 50.0  &  48.0  &  28.0  &   \\ 
			14 & 0.13  &  0.49  &  1.47  &  3.76  &  8.31  &  16.1  &  27.4  &  40.4  &  \bf 50.0  &  47.7  &  27.7  &   \\ 
			15 & 0.19  &  0.63  &  1.75  &  4.21  &  8.92  &  16.8  &  28.0  &  40.7  &  \bf 49.9  &  47.5  &  27.5  &   \\ 
			16 & 0.26  &  0.78  &  2.01  &  4.61  &  9.46  &  17.3  &  28.5  &  41.0  &  \bf 49.9  &  47.2  &  27.4  &   \\ 
			17 & 0.33  &  0.92  &  2.26  &  4.98  &  9.93  &  17.8  &  28.9  &  41.2  &  \bf 49.8  &  47.0  &  27.2  &   \\ 
			18 & 0.4  &  1.05  &  2.48  &  5.31  &  10.3  &  18.3  &  29.2  &  41.3  &  \bf 49.7  &  46.8  &  27.1  &   \\ 
			\vdots & & & & & & & & & & & \\ 
			10^5 & 2.73  &  4.49  &  7.22  &  11.3  &  17.1  &  24.9  &  34.2  &  43.3  &  \bf 48.4  &  43.7  &  25.0  &   \\ 
			\end{matrix} \]
\caption{Percentage of $k$-winnable $321$-avoiding permutations of $N$}\label{f:pcts}
\end{figure}

\begin{lemma}\label{l:del}
Let $i \leq N-5$ and $X_i$ be a linear combination of the Catalan numbers $C_{N-1}, C_{N-2}, \cdots, C_{N-i}$.  Then,
\[ \frac{1}{4} \frac{X_i}{C_N} < \frac{\Delta X_i}{C_N} \leq \frac{1}{3} \frac{X_i}{C_N}. \]
\end{lemma}
\begin{proof}
Observe that $\frac{1}{4} < \frac{C_{N-1}}{C_N} \leq \frac{1}{3}$ for all $N \geq 5$.  Since
$\frac{\Delta C_{N-i}}{C_N} = \frac{C_{N-i-1}}{C_N} = \frac{C_{N-i-1}}{C_{N-i}}
\frac{C_{N-i}}{C_N}$
we have
\[ \frac{1}{4} \frac{C_{N-i}}{C_N} < \frac{\Delta C_{N-i}}{C_N} \leq \frac{1}{3} \frac{C_{N-i}}{C_N} \]
for all $N-i \geq 5$, and the result follows by linearity.
\end{proof}

\begin{lemma}
For all $i \leq N-5$, we have 
\[ \frac{T_i}{C_N} \leq \frac{1}{3} \left(\frac{3}{4}\right)^{i-1} \]
\end{lemma}
\begin{proof}
It is straightforward to verify that the result holds for $i = 1$ and $i = 2$.  Suppose the result holds for $i-1$.  Then,
\[ \frac{T_i}{C_N} = \frac{T_{i-1}}{C_N} - \frac{\Delta T_{i-2}}{C_N} < \frac{T_{i-1}}{C_N} - \frac{1}{4} \frac{T_{i-2}}{C_N} \]
by Lemma~\ref{l:del}.  From their definition in terms of lattice paths, it is
also clear that the $T_i$ are decreasing in $i$ (for each fixed $N$).  Hence,
\[ \frac{T_{i-1}}{C_N} - \frac{1}{4} \frac{T_{i-2}}{C_N} \leq
\frac{T_{i-1}}{C_N} - \frac{1}{4} \frac{T_{i-1}}{C_N} = \frac{3}{4}
\frac{T_{i-1}}{C_N} \leq \frac{1}{3} \left(\frac{3}{4}\right)^{i-1} \]
by induction.
\end{proof}

\begin{theorem}\label{t:3m}
We have $\frac{S_3}{C_N} > \frac{S_i}{C_N}$ for all $N \geq 9$ and all $i > 3$.
\end{theorem}
\begin{proof}
We have
\[ \frac{S_i}{C_N} = \frac{i T_i + \Delta S_{i-1}}{C_N} \leq \frac{i}{3} \left(\frac{3}{4}\right)^{i-1} + \frac{1}{3} \frac{S_{i-1}}{C_N}. \]
An exercise using calculus proves that $\frac{i}{3}
\left(\frac{3}{4}\right)^{i-1}$ is decreasing once $i > -1/\ln(3/4)$ (which is
between $3$ and $4$) and that $\frac{i}{3} \left(\frac{3}{4}\right)^{i-1}$ is
less than $1/4$ for all $i \geq 11$.  Consequently, once $\frac{S_i}{C_N} <
\frac{3}{8}$, it remains so as $i$ increases, for all $i \geq 11$.  

In fact, using the linear combinations of Catalan numbers obtained from
Theorem~\ref{t:rec} as in Example~\ref{e:forms}, we can verify that
$\frac{S_i}{C_N} < \frac{3}{8}$ for all $5 \leq i \leq 11$ as illustrated in
Figure~\ref{f:pcts}.  More precisely, when we express $\frac{S_i}{C_N}$ as a
linear combination of ratios of Catalan numbers, the limiting value as $N
\rightarrow \infty$ can be obtained by plugging in powers of $1/4$ for each
ratio of Catalan numbers; as these limits are each smaller than $3/8$, we reduce to
a finite computation.  In detail, we use the bounds $0.25^j <
\frac{C_{N-j}}{C_N} < 0.254^j$ for $N > 95+j$ to verify that $\frac{S_i}{C_N} <
3/8$ for each of the linear combinations $i = 5, 6, \ldots, 11$ (and check
remaining finite cases for $N$ manually).

Thus, the optimal value of $\frac{S_i}{C_N}$ must occur in $i \leq 4$ for all
$N$.  Using the formulas from Example~\ref{e:forms} again, we then find that
$\frac{S_1}{C_N}$ is optimal for $N = 2$, that $\frac{S_2}{C_N}$ is optimal for
$3 \leq N \leq 8$, and that $\frac{S_3}{C_N}$ is optimal for all $N \geq 9$.
\end{proof}

\begin{corollary}
The optimal $k$-positional strategy for the game of best choice restricted to the
$321$-avoiding interview rank orders is
\[ k = \begin{cases}
        N-1 & \text{ if $N = 2$ } \\ 
        N-2 & \text{ if $3 \leq N \leq 8$ } \\ 
        N-3 & \text{ otherwise. } \\ 
\end{cases} \]
The asymptotic probability of success is 
\[ \lim_{N \rightarrow \infty} \frac{3C_{N-1}-4C_{N-2}-C_{N-3}}{C_N} = \frac{31}{64} = 0.484375. \]
\end{corollary}

Using Andr\'e's reflection method or a straightforward induction argument, one can show that the number of partial Dyck paths (i.e. lying above the line $y = x$) from $(0,0)$ to $(a,b)$ (where $a<b$) is given by the formula
\[ C_{(a,b)} = {{a+b} \choose a}\frac{b-a+1}{b+1}. \]
Using this, we can also give a direct count of the Dyck paths for which column
$k$ lies between the last two northeast corners of the path.  
\begin{theorem}
The probability that a 321-avoiding permutation of length $N$ is $k$-winnable is
\[ \frac{1}{C_N} \sum_{i=1}^{N-k} { {(k-1)+(N-i)} \choose {k-1} } \frac{(N-k-i+2)}{(N-i)+1}(N-k-i+1). \]
\end{theorem}
\begin{proof}
Set $a = k-1$, and let $b$ range over $k, k+1, k+2, \ldots, N-1$.  Once the
path passes through $(a,b)$, there are $b-k+1$ ways to complete it so that it
is $k$-winnable.
\end{proof}

\bigskip
\section{Conclusions}

It seems fair to say that these results are somewhat surprising and further
investigation is warranted.  The ``bar-raising'' model has a robust strategy
but only allows a $25\%$ success rate.  The optimal strategy in the
``disappointment-free'' model reviews and rejects most of the applicants yet
has a success rate that is close to $50\%$.   Remarkably, these are not
mutually exclusive and the $k = N-3$ positional strategy is asymptotically
optimal in both models simultaneously.

\bigskip
\section*{Acknowledgements}

This project was supported by the James Madison University Program of Grants for
Faculty Assistance.  We are grateful to an anonymous referee for several
insightful suggestions on an earlier draft of this work.


\providecommand{\bysame}{\leavevmode\hbox to3em{\hrulefill}\thinspace}
\providecommand{\MR}{\relax\ifhmode\unskip\space\fi MR }
\providecommand{\MRhref}[2]{%
  \href{http://www.ams.org/mathscinet-getitem?mr=#1}{#2}
}
\providecommand{\href}[2]{#2}

\end{document}